\documentclass[11pt]{amsart}

\usepackage{amsmath,amssymb,amsthm}
\usepackage{graphicx}
\usepackage{cite}
\usepackage[usenames, dvipsnames]{color}
\usepackage{dsfont}
\usepackage{mathrsfs}

\newtheorem{thm}{Theorem}[section]
\newtheorem{lemma}[thm]{Lemma}
\newtheorem{conj}[thm]{Conjecture}
\newtheorem{prop}[thm]{Proposition}
\newtheorem{cor}[thm]{Corollary}

\theoremstyle{definition}

\theoremstyle{remark}

\numberwithin{equation}{section}

\normalsize
\newcommand\numberthis{\addtocounter{equation}{1}\tag{\theequation}}

\def\eps{\varepsilon}
\def\N{\mathbb{N}}
\def\F{\mathbb{F}}

\def\b1{\mathds{1}}
\def\Z{\mathbb{Z}}
\def\C{\mathbb{C}}
\DeclareMathOperator\spa{span}
\DeclareMathOperator\rank{rank}
\DeclareMathOperator{\Ima}{Im}

\def\E{\mathbb{E}}
\def\Pr{\mathbb{P}}

\begin{document}

\title{On Szemer\'{e}di's theorem with differences from a random set}

\author{Daniel Altman}
\address{Mathematical Institute, Radcliffe Observatory Quarter, Woodstock Rd, Oxford, OX2 6GG, United Kingdom}
\email{daniel.h.altman@gmail.com}


\date{\today}

\begin{abstract}
We consider, over both the integers and finite fields, Szemer\'{e}di's theorem on $k$-term arithmetic progressions where the set $S$ of allowed common differences in those progressions  is restricted and random. Fleshing out a line of enquiry suggested by Frantzikinakis et al, we show that over the integers, the conjectured threshold for $\Pr(d \in S)$ for Szemer\'{e}di's theorem to hold a.a.s follows from a conjecture about how so-called dual functions are approximated by nilsequences. We also show that the threshold over finite fields is different to this threshold over the integers. 
\end{abstract}


\maketitle
\section{Introduction}\label{s:intro}
\subsection{Notation and definitions}
For a positive integer $N$, let $[N]$ denote the set $\{1, \ldots, N\}$. 

For $\delta >0$, a subset of $A$ of $[N]$ (respectively $\F_p^n$) will be said to have \textit{$\delta$-positive density} (or be \textit{$\delta$-dense}) if $|A| \geqslant \delta N$ (respectively $\geqslant \delta p^n$). A subset $B$ of $\N$ will be said to have \textit{positive upper density} if $\limsup_N |B \cap [N]|/N > 0$.

In a subset $A$ of an abelian group , a \textit{$k$-term arithmetic progression in $A$} (also \textit{$k$AP}) is a pair $(x,d)$ such that $x, x+d, \ldots, x+(k-1)d \in A$. For  $S \subset \N$, a $k$-term arithmetic progression has \textit{common difference in $S$} if, in the above notation, $d \in S$. 

When the ambient set is $[N]$ (respectively $\F_p^n$), we will say that \textit{Szemer\'{e}di's theorem with common differences in $S$ holds} if, for all $\delta>0, k\geqslant 2, N>N_0(k,\delta)$ (respectively $n>n_0(k,\delta)$) and sets $A\subset [N]$ (respectively $\subset \F_p^n$) of $\delta$-positive density, there exists a non-trivial $k$-term arithmetic progression in $A$ with common difference in $S$. When the ambient set is $\N$, \textit{Szemer\'{e}di's theorem with common differences in $S$ holds} means that all sets $B \subset \N$ with positive upper density contain a non-trivial $k$-term arithmetic progression with common difference in $S$.  

For a finite set $T$, we use the notation $\E_{x\in T}$ to denote the average over $T$, that is, $\frac{1}{|T|} \sum_{x \in T}$. We will often suppress the set $T$ and write $\E_x$ when the ambient set for $x$ is clear from context.

\subsection{Context}
In 1953, Roth \cite{Rot53} showed that sets of integers with positive upper density contain non-trivial 3-term arithmetic progressions. The result was famously extended to arbitrarily long arithmetic progressions by Szemer\'{e}di in 1975 \cite{Sze75}. It is well known that this is equivalent to the finitary formulation which asserts that, for $N$ sufficiently large in terms of $k$ and $\delta$, all $\delta$-dense subsets of $[N]$ contain non-trivial $k$APs.  

A natural generalization is to consider under what conditions Szemer\'{e}di's theorem is true when the set $S$ of allowed common differences in arithmetic progressions is restricted. It transpires that Szemer\'{e}di's theorem holds with common differences restricted to some fairly sparse sets $S \subset \N$; for example, a result of Bergelson and Leibman \cite{BL96} says that $S= \{1^{100}, 2^{100}, 3^{100},  \ldots \}$ (or indeed $\{p(n): n\in \N\}$ for an integer polynomial $p$ with $p(0)=0$) is sufficient. Of course, the set $S$ under consideration by Bergelson and Leibman is of a special structure. 

We are interested in the situation where $S$ is chosen at random. In the finitary model, it is common practice to construct the random set $S$ by selecting each $d$ to lie in $S$ independently with equal probability. In $\N$, the probability that $d$ lies in $S$ must be a function of $d$.

For $2$APs, it is known \cite{Bou87} that if $\Pr(d\in S) = \omega(\log N / N)$ then Szemer\'{e}di's theorem with common difference in $S$ holds asymptotically almost surely (a.a.s), and conversely that if $\Pr(d\in S) \leqslant C \log N / N$ then Szemer\'{e}di's theorem with common difference in $S$ a.a.s fails. For $k$APs, the current best result is due to Bri{\"e}t and Gopi in \cite{BG18}, which states that $\Pr(d \in S) = \omega\left( \frac{\log N}{N^{1/\left\lceil k/2 \right \rceil}}\right)$ is sufficient. In the case that $k=3$, this bound does not improve upon earlier work of Christ and Frantzikinakis, Lesigne and Wierdl in \cite{Chr11} and \cite{FLW12} respectively. 

For $k \geq 3$, there is a substantial gap between these results and conjectures found in work of Frantzikinakis and others. We include a reformulation here for convenience.
\begin{conj}[{\cite[Problem 31]{Fra16}, \cite[Conjecture 2.5]{FLW16}}]\label{c:franz}
Let $S \subset \N$ be chosen at random with $\Pr(d\in S) = \omega(1/d)$. Then it is almost surely the case that all subsets of $\N$ with positive upper density contain a $k$-term arithmetic progression with common difference in $S$.
\end{conj}

Conjecture \ref{c:franz} is in fact best possible in the sense that if  $S$ is constructed with $\Pr(d \in S) = 1/d$, then Szemer\'{e}di's theorem with common difference in $S$ fails (see discussion in \cite[Section 2]{FLW16}).

\subsection{Our results}
We study Szemer\'{e}di's theorem with differences restricted to random subsets of $[N], \N$ and $\F_p^n$. Although many of our methods generalize straightforwardly, we will focus on $k=3$ as much still remains to be understood about this special case.

In Section \ref{s:integers} (over $[N]$), by analogy to the case $k=2$, Conjecture \ref{c:conj} stipulates that so-called dual functions 
\[ F_A(d) := \E_x 1_A(x) 1_A(x+d) 1_A(x+2d),\]
for dense sets $A$ are well approximated by 2-step nilsequences. (This conjecture is similar to \cite[Special Case of Problem 1]{Fra16} - see \cite[Problem 1]{Fra16} for a discussion and related results). We show that under Conjecture \ref{c:conj}, Szemer\'{e}di's theorem in $[N]$ with common difference in $S$ a.a.s holds under two different probability models. Firstly, in Theorem \ref{t:intBound}, we choose $d$ to lie in $S$ with probability $\omega(\log N /N)$. Next, in Theorem \ref{t:logIntBound}, we choose $d$ to lie in $S$ with probability $\omega(1/d)$. The latter result is used to establish (almost surely) Szemer\'{e}di's theorem in $\N$ with common difference in $S$ where $\Pr(d \in S) = \omega(1/d)$ (Corollary \ref{cor:infinitary}). Thus, our Conjecture \ref{c:conj} implies Conjecture \ref{c:franz} above (focusing on the case $k=3$).

In Section \ref{s:finiteFields} we show (Corollary \ref{c:main}) that the analogous result to Theorem \ref{t:intBound} over finite fields is false (by some margin). Indeed, if $S$ is formed by selecting elements with probability 
\[\Pr(d \in S) = \frac{ c n^2}{p^n}, \]
 with $c = \frac{1}{2} - o(1) $, then Szemer\'{e}di's theorem for $k=3$ with common difference in $S$ almost surely fails.  
 We contrast this behavior with the case $k=2$, where the threshold for $\Pr(d\in S)$ for Szemer\'{e}di's theorem in $[N]$ to hold is analogous to the threshold over finite fields.

\section{Over the integers}\label{s:integers}

The result that $\Pr(d \in S) = \omega(\log N / N)$ is sufficient for Szemer\'{e}di's theorem on 2APs to a.a.s hold in $[N]$ (see e.g. \cite{Bou87}) can be proven by considering the 2-dual functions 
\[ F_A^{(2)} (d) := \E_x 1_A(x) 1_A(x+d),\]
which count the average number of $2$APs in dense sets $A$ with common difference $d$. Then one is interested in the quantity 
\begin{align*}
\langle F_A^{(2)} , 1_S \rangle := &\E_d \left( \E_{x}  1_A(x) 1_A(x+d)  \right)1_S(d)\\
= &\E_{x,d}  1_A(x) 1_A(x+d) 1_S(d),
\end{align*}
which counts the average number of 2APs in $A$ with common difference in $S$. The functions $F_A^{(2)}$ are just convolutions and so, by Fourier analysis, are well-approximated by convex combinations of additive characters selected uniformly in $A$. Then, roughly speaking, one can control the quantity $\langle F_A^{(2)}, 1_S \rangle$ uniformly in $A$ by controlling $\langle \phi, 1_S \rangle$ across a set of additive characters $\{\phi\}$. 

In pursuit of a similar argument for the case $k=3$, we define
\[ F_A (d) := \E_x 1_A(x) 1_A(x+d) 1_A(x+2d),\]
and provide Conjecture \ref{c:conj}, that these dual functions are well-approximated by convex combinations of 2-step nilsequences.

See \cite[Problem 1]{Fra16} for similar conjectures and related discussion. See also Appendix A for details on what we mean by a nilsequence and its complexity. 

\begin{conj}\label{c:conj}
Let $\eps > 0$. There exists a set of 2-step nilsequences $\{\phi_j\}$ of complexity $O_\eps(1)$, and, for every $A$, a family of coefficients $c_{A,j}$ with $\sum_j| c_{A,j}|  \leqslant 1$ and a family of error functions $e_A$ with $\left| e_A(d) \right| \leqslant \eps$ for $d=1,\ldots,N$, such that 
\[ F_A(d) = \sum_j c_{A,j} \phi_j(d) + e_A(d). \]
\end{conj}

The difficulty in proving Conjecture \ref{c:conj} lies in controlling the error $e_A$ in $L^\infty$-norm. For example, if we only demanded control of $e_A$ in $L^2$-norm, then the above decomposition would follow as a consequence of the Koopman--von Neumann theorem.

The following proposition (known to experts) will yield that it suffices to consider only polynomially-many nilsequences. We defer to Appendix A for a proof. 

\begin{prop}\label{p:nilEntropy}
Let $s,\eps, C>0$.  There is a set $X_{s,\eps,C}$ of $N^{O_{s,\eps, C}(1)}$ $s$-step nilsequences of complexity $O_{s,\eps, C}(1)$ such that, for any $s$-step nilsequence $\phi$ of complexity at most $C$, there is some $\phi^\prime \in X_{s,\eps, C}$ with 
$\left|  \phi(d) - \phi^\prime(d) \right| \leqslant \eps$
for $d=1,\ldots,N$. 
\end{prop}

Combining Conjecture \ref{c:conj} with Proposition \ref{p:nilEntropy} yields the following corollary.

\begin{cor}\label{cor:polyConj}
Conjecture \ref{c:conj} holds if and only if it holds under the restriction that the cardinality of the set of 2-step nilsequences is $N^{O_\eps(1)}$. 
\end{cor}

\begin{thm}\label{t:intBound}
Let $S \subset [N]$ be formed by letting each $d \in [N]$ lie in $S$ independently with probability $\sigma := \omega(\log N/N)$. If Conjecture \ref{c:conj} holds, then Szemer\'{e}di's theorem for $k=3$ with common differences in $S$ holds with probability $1-o(N^{-100})$.
\end{thm}

\begin{proof}
For $d=1,\ldots, N$, let $Y(d)$ denote the mean-zero random variable $1_S(d) - \sigma$. We will show that, for any $\eps > 0$,
\begin{equation}\label{e:toShow}
\Pr\left(\sup_A \left| \langle F_A, Y \rangle \right| \geqslant 6\eps \sigma \right) = o(N^{-100}).
\end{equation} 
Then, with probability $1 - o(N^{-100})$, we have that 
\[\E_{x,d}1_A(x)1_A(x+d)1_A(x+2d)1_S(d) > \sigma \E_{x,d} 1_A(x) 1_A(x+d) 1_A(x+2d) - 6\eps \sigma \]
uniformly in $A$. In particular, by Varnavides' theorem \cite{Var59}, the average number of 3APs in $A$ (that is, $\E_{x,d} 1_A(x) 1_A(x+d) 1_A(x+2d)$) is bounded away from zero (uniformly in $N$). Sending $\eps \to 0$ yields that, with probability $1-o(N^{-100})$, 
\[\E_{x,d}1_A(x)1_A(x+d)1_A(x+2d)1_S(d) > 0, \]
uniformly in $A$.

It remains to show (\ref{e:toShow}). To this end, let $\eps>0$ and induce Conjecture \ref{c:conj} and Corollary \ref{cor:polyConj} to write 
\[ F_A = \sum_j c_{A,j}\phi_j + e_A,\]
where the sum is over $N^{O_\eps(1)}$ nilsequences of complexity $O_\eps(1)$. Then,
\begin{align*}
\Pr\left(\sup_A \left| \langle F_A, Y \rangle \right| \geqslant 6\eps \sigma \right) &\leqslant \Pr\left(\sup_A \left( \sum_j \left|c_{A,j}\langle \phi_j, Y \rangle \right| + \left| \langle e_A, Y \rangle \right| \right) \geqslant 6\eps \sigma \right)\nonumber \\
&\leqslant  \Pr\left(\sup_j  \left|\langle \phi_j, Y \rangle \right|  \geqslant 3\eps \sigma \right) \\ & \quad +  \Pr\left(\sup_A  \left| \langle e_A, Y \rangle \right| \geqslant 3\eps \sigma \right).\numberthis \label{e:split}
\end{align*}
We will work on each of these terms separately. 

Firstly, by the union bound, we have that 
\begin{equation}\label{e:unionUpper}
\Pr\left(\sup_j  \left|\langle \phi_j, Y \rangle \right|  \geqslant 3\eps \sigma \right) \leqslant N^{O_\eps(1)}\Pr \left( \left|\langle \phi, Y \rangle \right|  \geqslant 3\eps \sigma \right), \end{equation}
where $\phi$ is some 2-step nilsequence of complexity $O_\eps(1)$. In particular, $\phi$ is  bounded in terms of $\eps$; that is, $\left| \phi(d) \right| = O_\eps(1)$ for $d\in [N]$. Note then that the $\phi(d)Y(d)$ are independent mean zero random variables with variance $O_\eps(\sigma)$. By Bernstein's inequality (\cite{Ber46}), 
\begin{equation}\label{e:bern} \Pr \left( \left|\langle \phi, Y \rangle \right|  \geqslant 3\eps \sigma \right) \leqslant \exp\left(-C_\eps\sigma N \right).
\end{equation}
Combining this inequality with (\ref{e:unionUpper}), we have 
\begin{align} \label{e:mainBound}
\Pr\left(\sup_j  \left|\langle \phi_j, Y \rangle \right|  \geqslant 3\eps \sigma \right) &\leqslant \exp\left(O_\eps(\log N) -C_\eps\sigma N\right) \nonumber \\ 
&= o(N^{-100}),
\end{align}
since $\sigma = \omega(\log N / N)$.

For the error term, using again Bernstein's inequality in the penultimate line, we have 
\begin{align}\label{e:error}
\Pr \left(\sup_A  \left| \langle e_A, Y \rangle \right| \geqslant 3\eps \sigma \right) 
&\leqslant  \Pr \left(\eps \sum_{d=1}^N \left| Y(d) \right| \geqslant 3N\eps \sigma \right) \nonumber\\
&\leqslant  \Pr \left( \sum_{d=1}^N \left( \left| Y(d) \right| -2\sigma(1-\sigma) \right) \geqslant N \sigma (1+2\sigma) \right) \nonumber\\
&\leqslant \exp \left( -CN\sigma \right) \nonumber\\
&= o(N^{-100}).
\end{align}

Now, combining (\ref{e:split}), (\ref{e:mainBound}) and (\ref{e:error}), we have that 
\[\Pr\left(\sup_A|\langle F_A, Y \rangle | \geqslant 6\eps \sigma \right) = o(N^{-100}),\]
when $\sigma = \omega(\log N /N)$, completing the proof.
\end{proof}

\begin{thm}\label{t:logIntBound}
Let $S \subset [N]$ be formed by letting each $d \in [N]$ lie in $S$ independently with probability $\mu(d) = \omega(1/d)$. If Conjecture \ref{c:conj} holds, then Szemer\'{e}di's theorem for $k=3$ with common differences in $S$ holds with probability $1-o(N^{-100})$.
\end{thm}
\begin{proof}
The proof is essentially the same as that of Theorem \ref{t:intBound}. We will focus on points of the argument that differ.

We will show that for any $\eps > 0$,
\begin{equation}\label{e:toShow2}
\Pr\left(\sup_A \left| \langle F_A, Y \rangle \right| \geqslant 6\eps \sigma \right) = o(N^{-100}),
\end{equation}
where this time $Y(d) = 1_S(d) - \sigma$, $1_S(d)$ is Bernoulli with parameter $\mu(d)=\omega(1/d)$ and $\sigma = \E_d \mu(d)$. Having established (\ref{e:toShow2}), it is easily checked that the rest of the proof is identical with this slightly-different definition of $\sigma$.

The main difference in establishing (\ref{e:toShow2}) is that now  the $Y(d)$ are not mean zero, and so we cannot conclude (\ref{e:bern}) immediately from Bernstein's inequality. However, one checks that 
\[\left| \E_d\E(\phi(d)Y(d)) \right| = o(\sigma), \]
and so modifying to the left hand side of (\ref{e:bern}) before inducing Bernstein's inequality yields a proof of (\ref{e:bern}). One also easily proves (\ref{e:error}) for our newly-defined $Y(d)$. We omit the details; the rest of the argument remains the same.   
\end{proof}

The following corollary extends the finitary result obtained in Theorem \ref{t:logIntBound} to a result in $\N$. It says, in particular, that the conjectures we inherited from \cite[Problem 31]{Fra16} and  \cite[Conjecture 2.5]{FLW16} are true for $k=3$ under our Conjecture \ref{c:conj}.

\begin{cor}\label{cor:infinitary}
Let $S \subset \N$ be chosen at random with $\Pr(n\in S) = \omega(1/n)$. Then, if Conjecture \ref{c:conj} holds, it is almost surely the case that all subsets of $\N$ with positive upper density contain a 3-term arithmetic progression with common difference in $S$.
\end{cor}
\begin{proof}
For $B \subset \N$, let $E_B$ be the event that $B$ contains a 3AP with common difference in $S$, and let $E_{B,N}$ be the event that $B \cap [N]$ contains a 3AP with common difference in $S \cap [N]$. Then $E_B = \bigcup_{N=1}^\infty E_{B,N}$. 

For $m=1,2,\ldots$, let $I_m = \{ B \subset \N : \limsup_N |B\cap [N]|/N \geqslant 1/m\}$,  and for each $N$ let $I_{m,N} = \{ B \subset \N : |B \cap [N]|/N \geqslant 1/m\}$. Observe that $I_m = \limsup_N I_{m,N}$.

Let $G_m$ be the event that all $B \in I_m$ contain a 3AP with common difference in $S$, that is $G_m = \cap_{B \in I_m} E_B$. Then 
the probability that Szemer\'{e}di's theorem with common difference in $S$ holds is given by $\Pr\left(\bigcup_{m=1}^\infty G_m \right)$.
By the monotone convergence theorem, this is equal to $\lim_{m\to \infty} \Pr(G_m)$. We will show that $\Pr(G_m^c ) = 0$ for all $m$. To this end we compute, 
\[ \Pr(G_m^c) = \Pr \left( \bigcup_{B \in \limsup_N I_{m,N}} E_B^c \right) \leqslant \Pr \left( \limsup_N \bigcup_{B \in I_{m,N}} E_{B,N}^c \right). \]

From Theorem \ref{t:logIntBound} it follows that $\Pr \left( \bigcup_{B \in I_{m,N}} E_{B,N}^c \right) = o(N^{-100})$, so that $\sum_{N=1}^\infty  \Pr \left( \bigcup_{B \in I_{m,N}} E_{B,N}^c \right) < \infty.$ Thus, by the Borel--Cantelli lemma, we have that $\Pr \left( \limsup_N \bigcup_{B \in I_{m,N}} E_{B,N}^c \right)  = 0$. The result follows. 
\end{proof}

\section{Over finite fields}\label{s:finiteFields}
Recall that for 2-term arithmetic progressions, $\Pr(d \in S) = \omega(\log N/N)$ was sufficient for Szemer\'{e}di's theorem with common difference in $S$ to hold asymptotically almost surely. It is not difficult to prove that the analogous fact is true over finite fields: that $\Pr(d \in S) = \omega(\log(p^n)/p^n)$ is sufficient. (One uses similar Fourier-analytic arguments to the $k=2$ case over the integers; here there is only a discrete set of additive characters so the argument is even easier.)

We showed in the previous section that, under Conjecture \ref{c:conj}, if elements of $[N]$ are chosen to lie in $S$ with probability $\omega(\log |[N ]|/ |[N]|)$ then Szemer\'{e}di's theorem for $k=3$ almost surely holds with common difference in $S$. In this section, we show that the analogous result over $\F_p^n$ is not true. In fact, if elements are selected to lie in $S$ independently with probability
\[\Pr(d \in S) = \frac{c n^2}{p^n} = \frac{c \log_p^2 |\F_p^n|}{|\F_p^n|} ,\]
where $c=1/2-o(1)$, then there will almost surely exist a set $A$ with positive density such that $A$ contains no 3APs with common difference in $S$. (Actually we deal with a slightly different probability model for convenience, but the above statement is an easy consequence of Corollary \ref{c:main}.)

The reason for the different behavior is that there are far more quadratic obstructions to 3APs in $\F_p^n$. Indeed, for $M\in M_n(\F_p)$, define $A_M = \{ x: x^\top M x = 0\}$ and note that $A_M$ has positive density (uniformly in $n$). One observes that if $x, x+d, x+2d \in A_M$, then $d^\top M d =0$, that is, $d\in A_M$. It follows that if all $A_M$ are to have 3APs with common difference in $S$, then $S$ must have the following property: for all $M \in M_n(\F_p)$, there exists $d\in S$ such that $d^\top M d = 0$.   

\begin{thm}\label{t:main}
Fix $p$ an odd prime. If $S\subset \F_p^n$ is formed by selecting at most $\binom{n+1}{2} - 11 n\log_p n$ elements of $\F_p^n$ independently at random,  then almost surely as $n\to \infty$ there exists some $M \in M_n(\F_p)$ such that $d^\top M d \ne 0$ for all $d\in S$.   
\end{thm}

\begin{cor}\label{c:main}
Fix $p$ an odd prime. If $S\subset \F_p^n$ is formed by selecting at most $\binom{n+1}{2} - 11 n\log_p n$ elements of $\F_p^n$ independently at random, then almost surely as $n\to \infty$ there exists some set $A \subset \F_p^n$ of positive density such that $A$ contains no 3-term arithmetic progression with common difference in $S$.
\end{cor}

The remainder of this section will prove Theorem \ref{t:main}. 

It suffices to consider symmetric matrices because if $M \in M_n(\F_p)$ then, letting $M^\prime = (M+M^\top )/2 \in S_n(\F_p)$, we have $d^\top M d = d^\top M^\prime d$ for all $d$.
We identify $S_n(\F_p)$ with $\F_p^{\binom{n+1}{2}}$ naturally; we will write $M_v$ for the matrix corresponding to a vector $v$, and $v_M$ for the vector corresponding to a matrix $M$. Define $\varphi: \F_p^n \to \F_p^{{\binom{n+1}{2}}}$ to be the degree 2 Veronese map, that is $(d_1,\ldots, d_n) 
\mapsto (d_i d_j)_{1 \leqslant i \leqslant j \leqslant n}$. Then $d^\top Md = \varphi(d)\cdot v_M$ and so $d^\top M d \ne 0$ for all $d\in S$ if and only if $v_M \not \in \cup_{d\in S} \varphi(d)^\perp$.

The following lemma demonstrates that if $\varphi(S)$ is linearly independent then there exists some matrix $M$ with $v_M \not \in \cup_{d\in S} \varphi(d)^\perp$. 

\begin{lemma}\label{l:linindHyps}
Let $\{v_1, \ldots, v_k\}$ be linearly independent in an $m$-dimensional vector space over $\F_p$. Then 
\[ \left| \bigcup_{i=1}^k v_i^\perp \right| = p^m\left( 1- \left( \frac{p-1}{p}\right)^k \right).\]
In particular, 
\[ \bigcup_{i=1}^k v_i^\perp  \subsetneq \F_p^m.\]
\end{lemma}
\begin{proof}
Linear algebra; we omit the details.
\end{proof}

The goal will now be to show that, almost surely as $n\to \infty$, the elements $\varphi(d), d\in S$ are linearly independent. Let $\mathcal{W}_k$ be the set of all $k$ dimensional subspaces of $\F_p^{\binom{n+1}{2}}$ and let $W_k \in \mathcal{W}_k$ be a subspace such that $|W_k \cap \Ima \varphi| = \max_{W \in \mathcal{W}_k} |W \cap \Ima \varphi|$. 

\begin{lemma}\label{l:probLinInd}
Let $N = |S|$. The probability that $\varphi(S)$ is linearly independent is bounded below by 
\[ \left( 1- \Pr_{d \in \F_p^n} \left( \varphi(d) \in W_N \right)\right)^{N}.\]
\end{lemma}
\begin{proof}
Sampling $S$ by selecting elements $d_1, \ldots, d_N$ at independently at random, the probability that $\varphi(S)$ is linearly independent is bounded below by,
 \begin{align*} 
 &\Pr(d_1 \ne 0)\prod_{i=2}^{N} \Pr(\varphi(d_i) \not \in \spa \{\varphi(d_1), \ldots, \varphi(d_{i-1})\} ) \\
\geqslant &\Pr(d_1 \ne 0)\prod_{i=2}^{N} \Pr(\varphi(d_i) \not \in  W_{i-1}) \\
\geqslant &\left( 1- \Pr_{d \in \F_p^n} \left( \varphi(d) \in W_N \right)\right)^{N}.
\end{align*}
\end{proof}

By Lemma \ref{l:probLinInd}, to show that $\varphi(S)$ is almost surely linearly independent, it suffices now to show that $\Pr_{d \in \F_p^n} \left( \varphi(d) \in W_N \right) = o(1/n^2)$. As an intermediate step, we will show in Proposition \ref{p:boundOnQuads} that 
\[ \Pr_{d \in \F_p^n} \left( \varphi(d) \in W_N \right) \leqslant \E_{{ v \in W_N^\perp}} p^{- \frac{1}{2} \rank M_v}. 
\] 
We separate out the main analytic observations in the following two lemmas. The first follows from orthogonality of characters.

\begin{lemma}\label{l:vectFuncts}
Let $V$ be a vector space of functions $\F_p^k \to \F_p$ under pointwise operations. Let $\omega = \exp\left( 2\pi i /p\right)$. Say $V(x) = 0$ if $v(x) = 0$ for all $v \in V$. Then 
\[ \Pr_x(V(x) = 0) = \E_{x,v} \omega^{v(x)}.\] 
\end{lemma}

The second is a standard estimate for Gauss sums.

\begin{lemma}\label{l:gaussSum}
Let $M$ be a symmetric matrix over $\F_p$. Then,
\[ \left| \E_x \omega^{x^\top Mx} \right| =p^{-\frac{1}{2}\rank M}.\]
\end{lemma}
\begin{proof}
The result follows from taking square roots after the following computation: 
\[ \left| \E_x \omega^{x^\top Mx} \right|^2 = \E_{x,h} \omega^{ (x+h)^\top M (x+h) - x^\top Mx} = \E_h \omega^{h^\top Mh} 1_{2Mh=0} = p^{-\rank M}.\]
\end{proof}

\begin{prop}\label{p:boundOnQuads}
We can bound $\Pr_d(\varphi(d) \in W_N)$ as follows: 
\[\Pr_d(\varphi(d) \in W_N) \leqslant \E_{M_v} p^{- \frac{1}{2} \rank M_v},\]
where the expectation is taken over all $M_v : v \in W_N^\perp$.
\end{prop}
\begin{proof}
Using Lemma \ref{l:vectFuncts} (with $V$ the vector space of quadratic forms defined by $\{M_v\}$) and Lemma \ref{l:gaussSum} we can compute that
\begin{align*}
\Pr_d(\varphi(d) \in W_N) &= \Pr_d(d^\top M_v d = 0 \text{ for all } v \in W_N^\perp)\\
&= \Pr_d(V(d) = 0) \\
&= \E_{d,M_v}\omega^{d^\top M_v d} \\
& \leqslant \E_{M_v} \left| \E_d \omega^{d^\top M_v d} \right| \\
& = \E_{M_v}p^{-\frac{1}{2}\rank M_v}.
\end{align*}
\end{proof}

Finally, it remains to show that
\[ \E_{M_v} p^{- \frac{1}{2} \rank M_v}  = o(1/n^2).\]
We do so with the following (crude) observations. Firstly, the number of $n \times n$ matrices of rank at most $r$ is bounded above by $p^{2nr}$ (choose the row space in at most $p^{nr}$ ways and then choose each of the $n$ rows in at most $p^r$ ways). Then, splitting the sum by rank, we have 
\begin{align*}
 \sum_{M_v} p^{- \frac{1}{2} \rank M_v}  &\leqslant \left| \{ M_v : \rank M_v < 5\log_p n\}\right| \\\ & \quad  + p^{-\frac{5}{2}\log_p n}\left|\{ M_v : \rank M_v \geqslant 5\log_p n\}\right| \\
 &\leqslant p^{10n\log_p n} + p^{-\frac{5}{2}\log_p n}\left| \{ M_v\} \right|,
\end{align*}
and so, recalling that $\left|\{M_v\}\right| = p^{\binom{n+1}{2} - N} = p^{11n\log_p n}$,
\[ \E_{M_v}  p^{- \frac{1}{2} \rank M_v}\leqslant p^{-n\log_p n} + p^{-\frac{5}{2}\log_p n}= o(1/n^2).\]

This completes the proof of Theorem \ref{t:main}.

\appendix 

\section{Some points on nilsequences}
We will briefly recall the main objects associated with nilsequences.  Our use of the term nilsequence essentially coincides with the definition of `polynomial nilsequence' in \cite[Definition 4.1]{GTZ12}. We direct an interested or concerned reader there for properly developed definitions and details.

The following definitions are essentially consistent with \cite[Definition 4.1]{GTZ12}. Throughout this section let $G$ be a connected, simply-connected, nilpotent Lie group with Lie algebra $\mathfrak{g}$. Let $\Gamma$ be a lattice (discrete, cocompact subgroup) in $G$, whence $G/ \Gamma$ is a nilmanifold. Let $G_\bullet = (G_i)_{i=0}^{s+1}$ be an $s$-step Lie filtration which is rational with respect to $\Gamma$ in the sense that $\Gamma \cap G_i$ is a lattice in $G_i$ for all $i$. Let $p: [N]\to G$ be a polynomial sequence with respect to $G_\bullet$. Let $\psi$ be a Lipschitz continuous function $G \to \C$ which is $\Gamma$-automorphic (we will often abuse notation and consider $\psi$ as a function on $G/\Gamma$). With this setup, defining $\phi(n):=\psi(p(n)\Gamma)$ yields a nilsequence.

We will also make a couple of minor amendments to \cite[Definition 4.1]{GTZ12}. Firstly, we will also add to the data associated to a nilsequence a Mal'cev basis $\mathscr{B}$ for the Lie algebra, which is adapted to the Lie filtration of $G$. This, in particular, is a basis with respect to which the structure constants of the Lie algebra $\mathfrak{g}$ are rational (the existence of such a basis is due to Mal'cev \cite{M49}). See \cite[Chapter 2]{GT12} and in particular \cite[Definition 2.1]{GT12} and the remarks that follow it for more details and discussion. Also, rather than use the left-invariant Riemannian metric on $G / \Gamma$ as in \cite[Definition 4.1]{GTZ12}, we will borrow the right-invariant metric $d_\mathscr{B}=d$ on $G / \Gamma$ from \cite[Definition 2.2]{GT12}.

We adopt a similar notion of the complexity of a nilsequence to that used in the formulation of the inverse conjecture for the Gowers $U^{s+1}[N]$ norm in \cite[Conjecture 4.5]{GTZ12}. When we refer to a nilsequence $\phi$ as having complexity bounded by $C>0$, we take as part of the definition that the following are also bounded by $C$:
\begin{itemize}
\item the dimension of $G$,
\item the heights of the (rational) structure constants of the Lie bracket operation  with respect to $\mathscr{B}$,
\item the heights of the (rational) coordinates of $\log(g_\Gamma)$ with respect to $\mathscr{B}$ for all $g_\Gamma \in S_\Gamma$, where $S_\Gamma$ is some generating set for $\Gamma$, 
\item the Lipschitz constant of $\psi$ (with respect to the metric $d$ in the domain), and
\item $\Vert \psi \Vert_\infty$.
\end{itemize} 

The goal of the remainder of the appendix is to prove Proposition \ref{p:nilEntropy}. Our first lemma is standard and follows from the definitions above.

\begin{lemma}\label{l:finNilman}
Let $C>0$ and let $\phi$ be an $s$-step nilsequence of complexity at most $C$. Then there are $O_{s,C}(1)$ possibilities for the Lie filtration $G_\bullet$ associated to $\phi$, and $O_{C}(1)$ possibilities for the lattice $\Gamma$. 
\end{lemma}

The following lemma is the key ingredient in the proof of Proposition \ref{p:nilEntropy}.

\begin{lemma}[{\cite[Lemma B.7]{BGSZ16}}]\label{l:polyPolyNilman}
Let $G / \Gamma$ be an $s$-step nilmanifold and let $\eps \in (0,1/2)$. There exists a set $P$ of $N^{O_{s,\eps} (1)}$ polynomial sequences $p':\Z \to G$ such that for every polynomial sequence $p$ in $G$ there exists $p' \in P$ with $d(p(n)\Gamma,p'(n)\Gamma) < \eps$ for all $n \in [N] $.
\end{lemma}

The final ingredient is a consequence of the Arzel{\` a}--Ascoli theorem, or can be verified from first principles.

\begin{lemma}\label{l:finLipschitz}
Let $\eps > 0$ and let $G/ \Gamma$ be a nilmanifold associated to a nilsequence as above.  Let $\mathcal{F}$ be the family of Lipschitz functions $G / \Gamma \to \C$ which are bounded by $C$ and have Lipschitz constant at most $C$. Then there exists a constant $K_{\eps,C}$ and a set $\mathcal{F}'$ of cardinality $O_{\eps,C}(1)$ containing Lipschitz functions $G / \Gamma \to \C$ which are bounded by $C$ and have Lipschitz constant at most $K_{\eps,C}$ with the following property. For every $\psi \in \mathcal{F}$, there exists $\psi' \in \mathcal{F}'$ with $|\psi(g) - \psi'(g)| < \eps$ for all $g \in G / \Gamma$.
\end{lemma}

\begin{proof}[Proof of Proposition \ref{p:nilEntropy}]
By Lemma \ref{l:finNilman}, there are $O_{s,C}(1)$ Lie groups, Lie filtrations and lattices corresponding to nilsequences of complexity at most $C$. Henceforth fix a particular nilmanifold $G/ \Gamma$. Let $\phi(n) = \psi(p(n)\Gamma)$ be a nilsequence of complexity at most $C$ with nilmanifold $G/ \Gamma$. Define $p'$ to be the polynomial sequence in $G$ produced by Lemma \ref{l:polyPolyNilman} with parameter $\frac{\eps}{2C}$. Also, define $\psi'$ to be the Lipschitz function $G / \Gamma \to \C$ produced by Lemma \ref{l:finLipschitz} with parameter $\eps/2$. Define the nilsequence $\phi'$ by $\phi'(n) = \psi'(p'(n)\Gamma)$. It follows that, 
\begin{align*}
|\phi(n) - \phi'(n)|  &= |\psi(p(n)\Gamma) - \psi'(p'(n)\Gamma)|\\
&\leqslant |\psi(p(n)\Gamma) - \psi(p'(n)\Gamma)| + |\psi(p'(n)\Gamma) - \psi'(p'(n)\Gamma)|\\
&< C d(p(n)\Gamma, p'(n)\Gamma) + \eps/2 \\
&< \eps,
\end{align*}
for all $n \in [N]$. Since we have chosen from a family of $N^{O_{s,\eps, C}(1)}$ such polynomial sequences $p'$ and $O_{\eps,C}(1)$ such Lipschitz functions $\psi'$, we have chosen $\phi'$ from a family of cardinality $N^{O_{s,\eps, C}(1)}$. It is also clear that these nilsequences have complexity $O_{s,\eps, C}(1)$. The result follows.
\end{proof}

\section*{Acknowledgments}
The author is indebted to his supervisor Ben Green for his guidance and support on this project. 

\bibliographystyle{alpha}
\nocite{*}
\bibliography{sze}

\end{document}